\documentclass{my_aims}

\usepackage{amsmath}
\usepackage{paralist}
\usepackage[colorlinks=true]{hyperref}

\hypersetup{urlcolor=blue, citecolor=red}

\newtheorem{theorem}{Theorem}[section]
\newtheorem{corollary}{Corollary}
\newtheorem{lemma}[theorem]{Lemma}
\newtheorem*{problem}{Problem $\mathbf{(P)}$}
\theoremstyle{definition}
\newtheorem{definition}[theorem]{Definition}
\newtheorem{remark}{Remark}

\title[Variational problems of Herglotz type]{Variational problems of Herglotz type
with time delay: DuBois--Reymond condition and Noether's first theorem}

\author[S. P. S. Santos, N. Martins and D. F. M. Torres]{}



\email{spsantos@ua.pt}
\email{natalia@ua.pt}
\email{delfim@ua.pt}

\thanks{Part of first author's Ph.D. project, which is carried out under the
Doctoral Programme in Mathematics (PDMat-UA) of University of Aveiro.}

\begin{document}

\maketitle


\centerline{\scshape Sim\~{a}o P. S. Santos, Nat\'{a}lia Martins and Delfim F. M. Torres}
\medskip
{\footnotesize
\centerline{Center for Research and Development in Mathematics and Applications (CIDMA)}
\centerline{Department of Mathematics, University of Aveiro, 3810--193 Aveiro, Portugal}
}


\begin{abstract}
We extend the DuBois--Reymond necessary optimality condition
and Noether's first theorem to variational problems of Herglotz
type with time delay. Our results provide, as corollaries,
the DuBois--Reymond necessary optimality condition and the first
Noether theorem for variational problems with time delay recently
proved in [Numer. Algebra Control Optim. 2 (2012), no.~3, 619--630].
Our main result is also a generalization of the first Noether-type theorem
for the generalized variational principle of Herglotz proved in
[Topol. Methods Nonlinear Anal. 20 (2002), no.~2, 261--273].
\end{abstract}


\section{Introduction}

It is well-known that the classical variational principle is a powerful tool
in various disciplines such as physics, engineering and mathematics.
However, the classical variational principle cannot describe many important physical processes.
In 1930 Herglotz \cite{Herglotz1930} proposed a generalized variational principle with one independent variable,
which generalizes the classical variational principle.
As reported in \cite{Georgieva2010,Georgieva2002,Georgieva2005},
the principle of Herglotz gives a variational description of nonconservative processes,
even when the Lagrangian is autonomous, something that cannot be done with the classical
approach. For the importance to include nonconservativism
in the calculus of variations, we refer the reader to \cite{book:FCV}.

The generalized variational calculus proposed by Herglotz deals with the following problem:
determine the trajectories $x\in C^{2}([a,b], \mathbb{R})$ satisfying given boundary
conditions $x(a)=\alpha$, $x(b)=\beta$, for fixed real numbers $\alpha$, $\beta$,
that extremize (minimize or maximize) the value
$$
z(b) \longrightarrow \mathrm{extr},
$$
where $z$ satisfies the differential equation
\begin{equation}
\label{eq_Herg}
\dot{z}(t)=L\left(t,x(t),\dot{x}(t),z(t)\right), \quad t \in [a,b],
\end{equation}
subject to the initial condition
\begin{equation}
\label{eq_initialvalue}
z(a)=\gamma,
\end{equation}
where $\gamma$ is a fixed real number.
The Lagrangian $L$ is assumed to satisfy the following hypotheses:
\begin{enumerate}
\item $L$ is a $C^1([a,b]\times\mathbb{R}^{3}, \mathbb{R})$ function;
\item functions  $\displaystyle t \mapsto \frac{\partial L}{\partial x}\left(t,
x(t), \dot{x}(t), z(t)\right)$, $\displaystyle t \mapsto \frac{\partial L}{\partial \dot{x}}\left(t,
x(t), \dot{x}(t), z(t)\right)$  and $\displaystyle t \mapsto
\frac{\partial L}{\partial z}\left(t, x(t), \dot{x}(t), z(t)\right)$,
 are differentiable  for any admissible trajectory $x$.
\end{enumerate}
Note that \eqref{eq_Herg} represents a family of differential equations: for each function $x$
a different differential equation arises. Therefore, $z$ depends on $x$, a fact that
can be made explicit by writing $z\{x;t\}$ (or $z(t, x(t), \dot{x}(t))$),
but for brevity and convenience of notation it is usual to write simply $z(t)$.
Observe that Herglotz's variational problem reduces to the classical fundamental problem
of the calculus of variations (see, e.g., \cite{GelfandFomin}) if the Lagrangian $L$
does not depend on the variable $z$: if
\begin{gather*}
\dot{z}(t)=L(t,x(t),\dot{x}(t)), \quad t \in [a,b],\\
z(a)=\gamma, \quad \gamma \in \mathbb{R},
\end{gather*}
then we obtain the classical variational problem
\begin{equation*}
z(b)=\int_a^b \tilde{L}(t,x(t),\dot{x}(t))dt \longrightarrow \mathrm{extr},
\end{equation*}
where
$$
\tilde{L}(t,x,\dot{x})=L(t,x,\dot{x})+\frac{\gamma}{b-a}.
$$
Herglotz proved that a necessary condition for a trajectory $x$
to be an extremizer of the generalized variational problem $z(b) \rightarrow \mathrm{extr}$
subject to \eqref{eq_Herg}--\eqref{eq_initialvalue} is given by
\begin{multline}
\label{eq:gen:EL:eq}
\frac{\partial L}{\partial x}\left(t,x(t),\dot{x}(t),z(t)\right)
-\frac{d}{dt}\frac{\partial L}{\partial \dot{x}}\left(t,x(t),\dot{x}(t),z(t)\right)\\
+\frac{\partial L}{\partial z}\left(t,x(t),\dot{x}(t),z(t)\right)
\frac{\partial L}{\partial \dot{x}}\left(t,x(t),\dot{x}(t),z(t)\right) = 0,
\end{multline}
$t \in [a,b]$. Herglotz called \eqref{eq:gen:EL:eq} \emph{the generalized
Euler--Lagrange equation} \cite{MR1391230,Guenther1996,MR1738100}.
Observe that for the classical problem of the calculus of variations
one has $\frac{\partial L}{\partial z}=0$, and the differential equation \eqref{eq:gen:EL:eq}
reduces to the classical Euler--Lagrange equation:
$$
\frac{\partial L}{\partial x}\left(t,x(t),\dot{x}(t)\right)
-\frac{d}{dt}\frac{\partial L}{\partial \dot{x}}\left(t,x(t),\dot{x}(t)\right) =0.
$$
A generalized Euler--Lagrange differential equation for Herglotz-type
higher-order variational problems was recently proved in \cite{Simao+NM+Torres2013}.

It is well-known that the notions of symmetry and conservation law play an important role in physics,
engineering and mathematics \cite{MR1901565}. The interrelation between symmetry and conservation laws
in the context of the calculus of variations is given by the first Noether theorem \cite{Noether1918}.
The first Noether theorem, usually known as Noether's theorem, guarantees that the
invariance of a variational integral under a group of transformations depending smoothly
on a parameter $\epsilon$ implies the existence of a conserved quantity along the
Euler--Lagrange extremals. Such transformations are global transformations.
Noether's theorem explains all conservation laws of mechanics, e.g.,
conservation of energy comes from invariance of the system under time translations;
conservation of linear momentum comes from invariance of the system under spatial translations;
while conservation of angular momentum reflects invariance with respect to spatial rotations.
The first Noether theorem is nowadays a well-known tool in  modern theoretical physics,
engineering and the calculus of variations \cite{MR2099056}.
Inexplicably, it is still not well-known that the famous paper of Emmy Noether \cite{Noether1918}
includes another important result: the second Noether theorem \cite{Torres-2003}.
Noether's second theorem states that if a variational integral has an infinite-dimensional Lie algebra of
infinitesimal symmetries parameterized linearly by $r$ arbitrary functions and their derivatives up to a given order $m$,
then there are $r$ identities between Euler--Lagrange expressions and their derivatives up to order $m$.
Such transformations are local transformations because can affect every part of the system differently.

In the last decades, Noether's theorems have been formulated in various contexts: see
\cite{Zbig+Nat+Torres2011,Zbig+Torres 2008,Cresson 2009,Gastao+Torres 2007 nonconservative,%
Gastao+Torres 2007,Gastao+Torres 2010 Riesz-Caputo,FredericoTorres2012,MR2671822,Malinowska2Nother,%
MalinowskaNM2013,MartinsTorres2010,Torres-2002,Torres-2003,MR2098297} and references therein.
Since the variational problem proposed by Herglotz defines a functional, whose extrema are sought,
by a family of differential equations, we cannot apply directly the important classical Noether's theorems
to this kind of variational problems. During more than 70 years, the generalization of the
two Noether theorems to variational problems of Herglotz type remained an open question.
The problem was solved in the beginning of XXI century by Georgieva and Guenther
\cite{Georgieva2001,Georgieva2010,Georgieva2002,Georgieva2005,Georgieva2003}.

The main goal of the present work is to generalize  Noether's first theorem
to variational problems of Herglotz type with time delay.
Variational problems with time delay arguments were first introduced
in 1964 by El'sgol'c \cite{El'sgol'c 1964},
who derived an Euler--Lagrange type condition for variational problems with time delay.
Since then, several authors have worked on various aspects of variational problems with time delay arguments
(see \cite{Agrawal et all 1997,MR3124697,Hughes-1968,Palm-1974,Sabbagh-1969} and references therein).
It has been proved that variational systems with time delay play an important role
in the modeling of phenomena in various applied fields.
However, only recently Frederico and Torres generalized the important Noether's
first theorem to optimal control problems with time delay \cite{FredericoTorres2012}.
Here we generalize Herglotz's problem by considering the following variational
problem with time delay. Throughout the text, $\tau$ denotes a real number such that $0\leq \tau<b-a$;
the partial derivative of $L$ with respect to its $i$th argument is denoted by $\partial_i L$.
To simplify notation, we write
$$
z\{x;t\}_\tau:=z(t,x(t),\dot{x}(t),x(t-\tau), \dot{x}(t-\tau)).
$$
\begin{problem}
Determine the trajectories $x\in C^{2}([a-\tau,b], \mathbb{R})$ satisfying given conditions
$$
x(b)=\beta \quad \mbox{and} \quad x(t)=\delta(t), \quad t\in [a-\tau, a],
$$
where $\beta$ is a  fixed real number and $\delta \in C^2([a-\tau,a], \mathbb{R})$ is a given initial function,
that extremize (minimize or maximize) the value of the functional $z\{x;b\}_\tau$,
$$
z(b) \longrightarrow \mathrm{extr},
$$
where $z$ satisfies the differential equation
\begin{equation}
\label{eq_Herg_delay}
\dot{z}(t)=L\left(t,x(t),\dot{x}(t),x(t-\tau),\dot{x}(t-\tau),z(t)\right),
\quad t \in [a,b],
\end{equation}
subject to the initial condition
\begin{equation}
\label{eq_initialvalue_Herg with delay}
z(a)=\gamma,
\end{equation}
where $\gamma$ is a fixed real number.
The Lagrangian $L$ is assumed to satisfy the following hypotheses:
\begin{description}
\item[$(H1)$] $L$ is a $C^1([a,b]\times\mathbb{R}^{5}, \mathbb{R})$ function;
\item[$(H2)$] functions $\displaystyle t \mapsto \partial_i
L\left(t,x(t),\dot{x}(t),x(t-\tau),\dot{x}(t-\tau),z(t)\right)$
for $i=2,\ldots, 6$ are differentiable for any admissible trajectory $x$.
\end{description}
\end{problem}
Observe that problem $z(b) \rightarrow \mathrm{extr}$
subject to \eqref{eq_Herg_delay}--\eqref{eq_initialvalue_Herg with delay}
reduces to the classical fundamental problem of the calculus of variations
with time delay if the Lagrangian $L$ does not depend on $z$.
Also note that problem $(P)$ reduces to the generalized
variational problem of Herglotz when $\tau=0$.

The structure of the paper is as follows. In Section~\ref{Preliminary results} we
review some preliminaries about the generalized variational calculus (without time delay).
In particular, we recall the notion of invariance and the first Noether theorem
for variational problems of Herglotz type. Our main results are given in Section~\ref{Main results}:
a generalized Euler--Lagrange necessary optimality condition (Theorem~\ref{Thm E-L eq time delay}),
a DuBois--Reymond necessary optimality condition (Theorem~\ref{Th. DuBois-R time delay}),
and Noether's first theorem for variational problems of Herglotz type with time delay
(Theorem~\ref{Thm Noether with delay}), are proved. We end with an illustrative example
of our results in Section~\ref{sec:ex}.
The results of the paper are trivially generalized for the case of vector functions
$x:[a-\tau, b]\rightarrow\mathbb{R}^n$, $n \in \mathbb{N}$,
but for simplicity of presentation we restrict ourselves to the scalar case.


\section{Review of Noether's first theorem for variational problems of Herglotz type}
\label{Preliminary results}

For the convenience of the reader, we present here the definition of generalized extremal,
the definition of invariance of functional $z$,
defined by $\dot{z}=L(t,x,\dot{x},z)$ and $z(a)=\gamma$, and we recall
Noether's first theorem for the generalized variational problem of Herglotz type.
For simplicity of notation, we introduce the operator
$\langle \cdot, \cdot \rangle$ defined by
$$
\langle x, z \rangle (t):= \left(t,x(t),\dot{x}(t),z(t)\right).
$$

\begin{theorem}[Generalized Euler--Lagrange equation \cite{Herglotz1930}]
\label{E-L thm Herg}
If $x \in C^{2}([a,b], \mathbb{R})$ is a solution to problem $z(b) \rightarrow \mathrm{extr}$
subject to \eqref{eq_Herg}--\eqref{eq_initialvalue} and the boundary conditions
$x(a)=\alpha$ and $x(b)=\beta$, for some fixed real numbers $\alpha, \beta$,
then $x$ satisfies the generalized Euler--Lagrange equation
\begin{equation}
\label{E-L eq}
\partial_2 L\langle x,z\rangle(t)
-\frac{d}{dt}\partial_3 L\langle x,z\rangle(t)
+\partial_4 L\langle x,z\rangle(t)
\partial_3 L\langle x,z\rangle(t)=0, \quad t\in [a,b].
\end{equation}
\end{theorem}

\begin{definition}[Generalized extremals---cf. \cite{Georgieva2002}]
The solutions $x \in C^{2}([a,b], \mathbb{R})$ of the
generalized Euler--Lagrange equation \eqref{E-L eq}
are called generalized extremals.
\end{definition}

Consider a one-parameter group
of infinitesimal transformations on $\mathbb{R}^{2}$,
\begin{equation}
\label{global rep}
\bar{t}=\phi(t,x,\epsilon), \quad \bar{x}=\psi(t,x,\epsilon),
\end{equation}
in which $\epsilon$ is the parameter and $\phi$ and $\psi$ are invertible $C^1$
functions such that $\phi(t,x,0)=t$ and $\psi(t,x,0)=x$.
The infinitesimal representation of transformations \eqref{global rep} is given by
\begin{equation*}
\begin{array}{c}
\bar{t}=t+\sigma(t,x)\epsilon+o(\epsilon),\\
\bar{x}=x+\xi(t,x)\epsilon+o(\epsilon),
\end{array}
\end{equation*}
where $\sigma$ and $\xi$ denote the first degree coefficients of $\epsilon$. Explicitly,
\begin{equation*}
\sigma(t,x)=\frac{d\phi}{d\epsilon}(t,x,\epsilon)\biggm\vert_{\epsilon=0},
\quad \quad \xi(t,x)=\frac{d\psi}{d\epsilon}(t,x,\epsilon)\biggm\vert_{\epsilon=0}.
\end{equation*}

\begin{definition}[Invariance---cf. Proposition 3.1 of \cite{Georgieva2002}]
\label{def invariance}
The one-parameter group of transformations
\[
\left\{
\begin{array}{c}
\bar{t}=\phi(t,x,\epsilon)\\
\bar{x}=\psi(t,x,\epsilon)
\end{array}
\right.
\]
leave the functional $z$, defined by $\dot{z}=L(t,x,\dot{x},z)$
and $z(a)=\gamma$ for some fixed real number $\gamma$, invariant, if
$$
\frac{d}{d\epsilon}\left[L\left(\bar{t},\bar{x}(\bar{t}),
\frac{d\bar{x}}{d\bar{t}}(\bar{t}),\bar{z}(\bar{t})\right)
\cdot \frac{d\bar{t}}{dt}\right] \biggm\vert_{\epsilon=0}=0.
$$
\end{definition}

We now prove the following useful result.

\begin{lemma}[Necessary condition for invariance]
\label{lemma invariance}
If the functional $z=z\{x;t\}$ defined by $\dot{z}(t)=L(t,x(t),\dot{x}(t), z(t))$
and $z(a)=\gamma$, for some fixed real number $\gamma$, is invariant under
the one-parameter group of transformations \eqref{global rep}, then
$$
\frac{d\bar{z}}{d\epsilon}(t)\biggm\vert_{\epsilon=0}=0
$$
for each $t \in [a,b]$.
\end{lemma}

\begin{proof}
Note that
\begin{equation*}
\frac{d\bar{z}}{d\bar{t}}(\bar{t})=L\left(\bar{t},\bar{x}(\bar{t}),
\frac{d\bar{x}}{d\bar{t}}(\bar{t}),\bar{z}(\bar{t})\right)
\end{equation*}
and by multiplying both sides of the equality by
$\displaystyle\frac{d\bar{t}}{dt}$ we have, by the chain rule, that
$$
\frac{d\bar{z}}{dt}(t)=\frac{d\bar{z}}{d\bar{t}}(\bar{t})
\cdot \frac{d\bar{t}}{dt}(t)= L\left(\bar{t},\bar{x}(\bar{t}),
\frac{d\bar{x}}{d\bar{t}}(\bar{t}),\bar{z}(\bar{t})\right)
\cdot \frac{d\bar{t}}{dt}(t).
$$
Now, differentiating with respect to $\epsilon$ and setting $\epsilon=0$,
we find, by definition of invariance, that
\begin{equation*}
\frac{d}{dt}\left(\frac{d\bar{z}}{d\epsilon}\right)\biggm\vert_{\epsilon=0}
= \frac{d}{d\epsilon}\left(\frac{d\bar{z}}{dt}\right)\biggm\vert_{\epsilon=0}
= \frac{d}{d\epsilon}\left[L\left(\bar{t},\bar{x}(\bar{t}),\frac{d\bar{x}}{d\bar{t}}(\bar{t}),
\bar{z}(\bar{t})\right)\cdot \frac{d\bar{t}}{dt}\right] \biggm\vert_{\epsilon=0}=0.
\end{equation*}
Defining $h(t):=\displaystyle\frac{d\bar{z}}{d\epsilon}(t)\big\vert_{\epsilon=0}$,
we get that $\displaystyle\frac{d h}{dt}(t)=0$ for all $t \in [a, b]$.
Since we are supposing  the initial condition $z(a)$ to be fixed ($z(a)=\gamma$), then
$\bar{z}(\bar{a})$ is also fixed ($\bar{z}(\bar{a})=\bar{\gamma}$) and hence $\frac{d}{d\epsilon}(\bar{z}(\bar{a}))\Big\vert_{\epsilon=0}=0.$
Observe that if $\bar{a}=a$, then $\displaystyle\frac{d\bar{z}}{d\epsilon}(a)\Big\vert_{\epsilon=0}=0$; if $\bar{a}\neq a$, then
\begin{equation*}
0=\frac{d}{d\epsilon}(\bar{z}(\bar{a}))\Big\vert_{\epsilon=0}
=\frac{d\bar{z}}{d\epsilon}(\bar{a})\Big\vert_{\epsilon=0}\frac{d\bar{a}}{d\epsilon}\Big\vert_{\epsilon=0}
=\frac{d\bar{z}}{d\epsilon}(a)\Big\vert_{\epsilon=0}\sigma(a,x)
\end{equation*}
and because $\sigma(a,x)\neq 0$, we can write that
$\displaystyle\frac{d\bar{z}}{d\epsilon}(a)\Big\vert_{\epsilon=0}=0$.
By definition of $h$, this means that $h(a)=0$.
Since $h$ is constant on $[a, b]$, we conclude that
$$
h(t):=\frac{d\bar{z}}{d\epsilon}(t)\biggm\vert_{\epsilon=0}=0
$$
for all $t \in [a,b]$.
\end{proof}

\begin{theorem}[Noether's first theorem for variational problems of Herglotz type \cite{Georgieva2002}]
\label{1st Noether THM}
If functional $z=z\{x;t\}$ defined by $\dot{z}=L\left(t,x(t),\dot{x}(t), z(t)\right)$
and $z(a)=\gamma$, for some fixed real number $\gamma$, is invariant under
the one-parameter group of transformations \eqref{global rep}, then
\begin{equation*}
\lambda(t) \cdot
\Bigl(\left[L\langle x,z\rangle(t)-\dot{x}\partial_3 L\langle x,z\rangle(t)\right]\sigma(t,x)
+\partial_3 L\langle x,z\rangle(t)\xi(t,x)\Bigr)
\end{equation*}
is conserved along the generalized extremals, where
$\lambda(t):= e^{-\int^t_a\partial_4 L \langle x, z \rangle(\theta)d\theta}$.
\end{theorem}


\section{Main results}
\label{Main results}

We prove some important results for variational problems of Herglotz type with time delay:
a generalized Euler--Lagrange necessary optimality condition (Theorem~\ref{Thm E-L eq time delay}),
a DuBois--Reymond necessary optimality condition (Theorem~\ref{Th. DuBois-R time delay})
and a Noether's first theorem for variational problems of Herglotz
type with time delay (Theorem~\ref{Thm Noether with delay}).
To simplify the presentation, we suppress most of the arguments
and the following notation is used throughout:
$$
[x, z]_\tau(t) := \left(t,x(t),\dot{x}(t), x(t-\tau),\dot{x}(t-\tau),z(t)\right).
$$

\begin{definition}[Admissible function]
A function $x\in C^2\left([a-\tau,b],\mathbb{R}\right)$ is said
to be admissible for problem $(P)$ if it satisfies the endpoint
condition $x(b)=\beta$ and $x(t)=\delta(t)$ for all $t \in [a-\tau,a]$.
\end{definition}

\begin{definition}[Admissible variation]
We say that $\eta\in C^2\left([a-\tau,b],\mathbb{R}\right)$
is an admissible variation for problem $(P)$
if $\eta(t)=0$ for $t \in [a-\tau,a]$ and $\eta(b)=0$.
\end{definition}

The following result gives a necessary condition of Euler--Lagrange type
for an admissible function $x$ to be an extremizer of the functional $z\{x;b\}_\tau$,
where $z$ is defined by \eqref{eq_Herg_delay}--\eqref{eq_initialvalue_Herg with delay}.

\begin{theorem}[Generalized Euler--Lagrange equations
for variational problems of Herglotz type with time delay]
\label{Thm E-L eq time delay}
If $x \in C^2([a-\tau,b],\mathbb{R})$ is a solution to problem $(P)$, then $x$ satisfies
the following generalized Euler--Lagrange equations with time delay:
\begin{multline}
\label{E-L delay 1}
\lambda(t+\tau)\biggl[\partial_4L[x, z]_\tau(t+\tau)
-\frac{d}{dt}\partial_5L[x, z]_\tau(t+\tau)\\
+\partial_5L[x, z]_\tau(t+\tau)\partial_6L[x, z]_\tau(t+\tau)\biggr]\\
+ \lambda(t)\left[\partial_2L[x, z]_\tau(t)
-\frac{d}{dt}\partial_3L[x, z]_\tau(t)+\partial_3
L[x, z]_\tau(t)\partial_6L[x, z]_\tau(t)\right]=0,
\end{multline}
$a\leq t \leq b-\tau$, where $\lambda(t):=
e^{-\int^t_a\partial_6L[x, z]_\tau(\theta) d\theta}$, and
\begin{equation}
\label{E-L delay 2}
\partial_2L[x, z]_\tau(t)-\frac{d}{dt}\partial_3L[x, z]_\tau(t)
+\partial_3L[x, z]_\tau(t)\partial_6L[x, z]_\tau(t)=0,
\end{equation}
$b-\tau \leq t \leq b$.
\end{theorem}

\begin{proof}
Suppose $x \in C^2([a-\tau,b],\mathbb{R})$ is a solution
to problem $(P)$ and let $\eta$ be an admissible variation.
Let $\epsilon$ be an arbitrary real number and define $\zeta:[a,b]\rightarrow\mathbb{R}$ by
$$
\zeta(t):=\frac{d}{d\epsilon}z\{x+\epsilon\eta;t\}_\tau\biggm\vert_{\epsilon=0}.
$$
Obviously, $\zeta(a)=0$ and, since $x$ is an extremizer, we conclude that $\zeta(b)=0$. Observe that
\begin{equation*}
\dot{\zeta}(t)= \frac{d}{dt}\frac{d}{d\epsilon}z\{x+\epsilon\eta;t\}_\tau\biggm\vert_{\epsilon=0}
=\frac{d}{d\epsilon}\frac{d}{dt} z\{x+\epsilon\eta;t\}_\tau\biggm\vert_{\epsilon=0}
=\frac{d}{d\epsilon}L[x+\epsilon \eta,z]_\tau(t)\biggm\vert_{\epsilon=0},
\end{equation*}
which means that
\begin{multline*}
\dot{\zeta}(t)
= \partial_2L[x,z]_\tau(t)\eta(t)+\partial_3L[x,z]_\tau(t)\dot{\eta}(t)
+ \partial_4L[x,z]_\tau(t)\eta(t-\tau)\\
+\partial_5L[x,z]_\tau(t)\dot{\eta}(t-\tau)+\partial_6L[x,z]_\tau(t)\zeta(t).
\end{multline*}
Consequently, $\zeta$ is solution
of the first order linear differential equation
$$
\dot{\zeta}=\partial_2L\eta(t)+\partial_3L\dot{\eta}(t)+ \partial_4L\eta(t-\tau)
+ \partial_5L\dot{\eta}(t-\tau)+\partial_6L\zeta.
$$
Hence, $\zeta$ satisfies the equation
\begin{multline*}
\lambda(t)\zeta(t)-\zeta(a)= \int_a^t\lambda(s)\biggl[\partial_2L[x,z]_\tau(s)\eta(s)
+\partial_3L[x,z]_\tau(s)\dot{\eta}(s)\\
+\partial_4L[x,z]_\tau(s)\eta(s-\tau)+ \partial_5L[x,z]_\tau(s)\dot{\eta}(s-\tau)\biggr]ds,
\end{multline*}
where
$\lambda(t):=\displaystyle e^{-\int^t_a\partial_6L[x, z]_\tau(\theta) d\theta}$.
The previous equation is valid for all $t \in [a,b]$, in particular for $t=b$.
Because $\zeta(a)=\zeta(b)=0$, we have
\begin{multline*}
\int_a^b\lambda(s)\left[\partial_2L[x,z]_\tau(s)\eta(s)
+\partial_3L[x,z]_\tau(s)\dot{\eta}(s) \right]ds\\
+ \int_a^b\lambda(s)\left[ \partial_4L[x,z]_\tau(s)\eta(s-\tau)
+ \partial_5L[x,z]_\tau(s)\dot{\eta}(s-\tau)\right]ds=0.
\end{multline*}
Applying the change of variable $s=t+\tau$ in the second integral
and recalling that $\eta$ is null in $[a-\tau, a]$, we obtain that
\begin{multline*}
\int_a^b\lambda(s)\left[\partial_2L[x,z]_\tau(s)\eta(s)
+\partial_3L[x,z]_\tau(s)\dot{\eta}(s)\right]ds\\
+ \int_a^{b-\tau} \lambda(s+\tau)\left[\partial_4
L[x,z]_\tau(s+\tau)\eta(s)+\partial_5L[x,z]_\tau(s+\tau)\dot{\eta}(s)\right]ds=0,
\end{multline*}
that is,
\begin{multline*}
\int_a^{b-\tau}\left[\lambda(s)\partial_2L[x,z]_\tau(s)
+ \lambda(s+\tau)\partial_4L[x,z]_\tau(s+\tau)\right]\eta(s) ds\\
+ \int_a^{b-\tau}\left[ \lambda(s)\partial_3L[x,z]_\tau(s)
+ \lambda(s+\tau)\partial_5L[x,z]_\tau(s+\tau) \right]\dot{\eta}(s)ds\\
+ \int_{b-\tau}^b\lambda(s)\left[\partial_2
L[x,z]_\tau(s)\eta(s)+\partial_3L[x,z]_\tau(s)\dot{\eta}(s)\right]ds=0.
\end{multline*}
Integration by parts gives
\begin{align*}
&\int_a^{b-\tau}\biggl\{ \lambda(s)\partial_2L[x,z]_\tau(s)
+ \lambda(s+\tau)\partial_4L[x,z]_\tau(s+\tau)\biggr.\\
&\qquad \biggl.-\frac{d}{ds}\left[ \lambda(s)\partial_3L[x,z]_\tau(s)
+ \lambda(s+\tau)\partial_5L[x,z]_\tau(s+\tau)\right]\biggr\}\eta(s)ds\\
&+\left[\left(\lambda(s)\partial_3L[x,z]_\tau(s)
+ \lambda(s+\tau)\partial_5L[x,z]_\tau(s+\tau)\right)\eta(s)\right]_a^{b-\tau}\\
&+ \int_{b-\tau}^b\left[\lambda(s)\partial_2L[x,z]_\tau(s)
-\frac{d}{ds}(\lambda(s)\partial_3L[x,z]_\tau(s))\right]\eta(s)ds\\
&+ \left[\lambda(s)\partial_3L[x,z]_\tau(s)\eta(s)\right]_{b-\tau}^b=0.
\end{align*}
Since previous equation holds for all admissible variations, it holds also
for those admissible variations $\eta$ such that $\eta(t)=0$
for all $t \in [b-\tau,b]$ and, therefore, we get
\begin{align*}
&\int_a^{b-\tau}\biggl\{ \lambda(s)\partial_2L[x,z]_\tau(s)
+ \lambda(s+\tau)\partial_4L[x,z]_\tau(s+\tau)\biggr.\\
&\qquad \biggl.-\frac{d}{ds}\left[ \lambda(s)\partial_3L[x,z]_\tau(s)
+ \lambda(s+\tau)\partial_5L[x,z]_\tau(s+\tau)\right]\biggr\}\eta(s)ds=0.
\end{align*}
From the fundamental lemma of the calculus of variations
(see, e.g., \cite{GelfandFomin}), we conclude that
\begin{multline*}
\lambda(t+\tau)\partial_4L[x,z]_\tau(t+\tau)+\lambda(t)\partial_2L[x,z]_\tau(t)\\
-\frac{d}{dt}\left[\lambda(t+\tau)\partial_5L[x,z]_\tau(t+\tau)+\lambda(t)\partial_3L[x,z]_\tau(t)\right]=0
\end{multline*}
for $a\leq t \leq b-\tau$, proving equation \eqref{E-L delay 1}.
Now, if we restrict ourselves to those admissible variations
$\eta$ such that $\eta(t)=0$ for all $t \in [a,b-\tau]$ we get
$$
\int_{b-\tau}^b\left[\lambda(s)\partial_2L[x,z]_\tau(s)
-\frac{d}{ds}(\lambda(s)\partial_3L[x,z]_\tau(s))\right]\eta(s)ds=0
$$
and from the fundamental lemma of the calculus of variations we conclude that
$$
\lambda(t)\partial_2L[x,z]_\tau(t)-\frac{d}{dt}(\lambda(t)\partial_3L[x,z]_\tau(t))=0
$$
for $b-\tau \leq t \leq b$, proving equation \eqref{E-L delay 2}.
\end{proof}

\begin{definition}[Generalized extremals with time delay]
\label{def invariance time delay}
The solutions $x\in C^2([a-\tau,b],\mathbb{R})$ of the
Euler--Lagrange equations \eqref{E-L delay 1}--\eqref{E-L delay 2}
are called generalized extremals  with time delay.
\end{definition}

\begin{remark}
Note that if there is no time delay, that is, if $\tau=0$, then problem $(P)$
reduces to the classical variational problem of Herglotz and
Theorem~\ref{E-L thm Herg} is a corollary
of our Theorem~\ref{Thm E-L eq time delay}.
\end{remark}

In order to simplify expressions, and in agreement with
Theorem~\ref{Thm E-L eq time delay}, from now on we use the notation
$$
\lambda(t):= e^{-\int^t_a \partial_6 L[x, z]_\tau(\theta) d\theta}.
$$
The following theorem gives a generalization of the
DuBois--Reymond condition  for classical variational problems \cite{Cesari1983} and generalizes
the Dubois--Reymond condition for variational problems with time delay of \cite{FredericoTorres2012}.

\begin{theorem}[DuBois--Reymond conditions for variational
problems of Herglotz type with time delay]
\label{Th. DuBois-R time delay}
If $x$ is a generalized extremal with time delay such that
\begin{equation}
\label{Extra Hip}
\partial_4L[x,z]_\tau(t+\tau)\cdot \dot{x}(t)
+\partial_5L[x,z]_\tau(t+\tau)\cdot\ddot{x}(t)=0
\end{equation}
for all $t\in [a-\tau, b-\tau]$, then $x$ satisfies the following equations:
\begin{multline}
\label{DuBois-R (1)}
\frac{d}{dt}\left\{\lambda(t) L[x,z]_\tau(t) - \dot{x}(t)\left[\lambda(t) \partial_3
L[x,z]_\tau(t) + \lambda(t+\tau)\partial_5L[x,z]_\tau(t+\tau)\right]\right\}\\
=\lambda(t)\partial_1L[x,z]_\tau(t)
\end{multline}
for $a \leq t\leq b-\tau$, and
\begin{equation}
\label{DuBois-R (2)}
\frac{d}{dt}\left\{\lambda(t)\left[L[x,z]_\tau(t)-\dot{x}(t)\partial_3L[x,z]_\tau(t)\right]\right\}
=\lambda(t)\partial_1L[x,z]_\tau(t)
\end{equation}
for $b-\tau \leq t \leq b$.
\end{theorem}

\begin{proof}
In order to prove equation \eqref{DuBois-R (1)},
let $t \in [a,b-\tau]$ be arbitrary. Note that
\begin{align*}
& \int_a^t \frac{d}{ds}\Bigl\{\lambda(s) L[x,z]_\tau(s)
- \dot{x}(s)\left[\lambda(s) \partial_3L[x,z]_\tau(s)
+ \lambda(s+\tau)\partial_5L[x,z]_\tau(s+\tau)\right]\Bigr\}ds\\
&= \int_a^t \Big\{ -\partial_6L[x,z]_\tau(s)\lambda(s) L[x,z]_\tau(s)
+\lambda(s)\Big[\partial_1L[x,z]_\tau(s)+\partial_2L[x,z]_\tau(s) \dot{x}(s)\\
&+ \partial_3L[x,z]_\tau(s)\ddot{x}(s)+\partial_4L[x,z]_\tau(s)\dot{x}(s-\tau)
+\partial_5L[x,z]_\tau(s)\ddot{x}(s-\tau)\\
&+ \partial_6L[x,z]_\tau(s)  L[x,z]_\tau(s)\Big] -\ddot{x}(s)\Big[\lambda(s)
\partial_3L[x,z]_\tau(s)+\lambda(s+\tau)\partial_5L[x,z]_\tau(s+\tau)\Big]\\
&-\dot{x}(s)\frac{d}{ds}\Big[\lambda(s) \partial_3L[x,z]_\tau(s)
+ \lambda(s+\tau)\partial_5L[x,z]_\tau(s+\tau)\Big] \Big\}ds.
\end{align*}
Canceling symmetrical terms, we get
\begin{align*}
& \int_a^t \frac{d}{ds}\left\{\lambda(s) L[x,z]_\tau(s)
- \dot{x}(s)\left[\lambda(s) \partial_3L[x,z]_\tau(s)
+ \lambda(s+\tau)\partial_5L[x,z]_\tau(s+\tau)\right]\right\}ds\\
&= \int_a^t \Big( \lambda(s)\partial_1L[x,z]_\tau(s)
+\lambda(s) \partial_2L[x,z]_\tau(s) \dot{x}(s)
-\ddot{x}(s)\lambda(s+\tau)\partial_5L[x,z]_\tau(s+\tau)\\
&-\dot{x}(s)\frac{d}{ds}\left[\lambda(s) \partial_3 L[x,z]_\tau(s)
+ \lambda(s+\tau)\partial_5L[x,z]_\tau(s+\tau)\right] \Big)ds \\
&+\int_a^t \Big(  \lambda(s)\partial_4L[x,z]_\tau(s)\dot{x}(s-\tau)
+ \lambda(s)\partial_5L[x,z]_\tau(s)\ddot{x}(s-\tau)\Big) ds.
\end{align*}
Observe that, by hypothesis \eqref{Extra Hip}, the last integral is null
and by substitution of the Euler--Lagrange equation \eqref{E-L delay 1} one gets
\begin{multline*}
\int_a^t \frac{d}{ds}\left\{\lambda(s) L[x,z]_\tau(s)
- \dot{x}(s)\left[\lambda(s) \partial_3L[x,z]_\tau(s)
+ \lambda(s+\tau)\partial_5L[x,z]_\tau(s+\tau)\right]\right\}ds\\
= \int_a^t \Big( \lambda(s)\partial_1L[x,z]_\tau(s)-\lambda(s+\tau)
\big[\partial_4L[x,z]_\tau(s+\tau) \dot{x}(s)
+\ddot{x}(s)\partial_5L[x,z]_\tau(s+\tau)\big]\Big)ds.
\end{multline*}
Using hypothesis \eqref{Extra Hip} in the right hand side
of the last equation, we conclude that
\begin{multline*}
\int_a^t \frac{d}{ds}\left\{\lambda(s) L[x,z]_\tau(s)
- \dot{x}(s)\left[\lambda(s) \partial_3L[x,z]_\tau(s)
+ \lambda(s+\tau)\partial_5L[x,z]_\tau(s+\tau)\right]\right\}ds \\
= \int_a^t \lambda(s)\partial_1L[x,z]_\tau(s) ds.
\end{multline*}
Condition \eqref{DuBois-R (1)} follows from the arbitrariness of $t\in [a,b-\tau]$.
In order to prove equation \eqref{DuBois-R (2)},
let $t \in[b-\tau, b]$ be arbitrary. Note that
\begin{align*}
& \int_t^b \frac{d}{ds}\left\{\lambda(s) L[x,z]_\tau(s)
- \lambda(s)\dot{x}(s)\partial_3 L[x,z]_\tau (s)\right\}ds\\
&= \int_t^b \Big\{ -\partial_6L[x,z]_\tau(s)\lambda(s) L[x,z]_\tau(s)
+\lambda(s)\Big[\partial_1L[x,z]_\tau(s)+\partial_2L[x,z]_\tau(s) \dot{x}(s)\\
&+\partial_3L[x,z]_\tau(s)\ddot{x}(s)+\partial_4L[x,z]_\tau(s)
\dot{x}(s-\tau)+\partial_5L[x,z]_\tau(s)\ddot{x}(s-\tau)\\
&+\partial_6L[x,z]_\tau(s)  L[x,z]_\tau(s)\Big]
-\ddot{x}(s)\lambda(s) \partial_3L[x,z]_\tau(s)
-\dot{x}(s)\frac{d}{ds}\left[\lambda(s)
\partial_3L[x,z]_\tau(s)\right] \Big\}ds.
\end{align*}
Cancelling symmetrical terms, the previous equation becomes
\begin{align*}
& \int_t^b \frac{d}{ds}\left\{\lambda(s) L[x,z]_\tau(s)
- \lambda(s)\dot{x}(s)\partial_3 L[x,z]_\tau (s)\right\}ds  \\
&= \int_t^b \Big\{\lambda(s)\big(\partial_1L[x,z]_\tau(s)+\partial_2L[x,z]_\tau(s) \dot{x}(s)\big)
- \dot{x}(s)\frac{d}{ds}\left[\lambda(s) \partial_3L[x,z]_\tau(s)\right] \Big\}ds\\
& \quad +\int_t^b \Big\{\lambda(s)\big(\partial_4L[x,z]_\tau(s)\dot{x}(s-\tau)
+\partial_5L[x,z]_\tau(s)\ddot{x}(s-\tau)\big)\Big\} ds.
\end{align*}
Substituting the Euler--Lagrange equation \eqref{E-L delay 2} and using
the hypothesis \eqref{Extra Hip} in the last integral, we conclude that
\begin{align*} & \int_t^b \frac{d}{ds}\left\{\lambda(s) L[x,z]_\tau(s)
- \lambda(s)\dot{x}(s)\partial_3 L[x,z]_\tau (s)\right\}ds
= \int_t^b \lambda(s)\partial_1L[x,z]_\tau(s) ds.
\end{align*}
Condition \eqref{DuBois-R (2)} follows from the arbitrariness of $t\in [b-\tau,b]$.
\end{proof}

\begin{remark}
For the classical variational problem and for the variational
problem of Herglotz (without delayed arguments), the hypothesis \eqref{Extra Hip}
is trivially satisfied. There is an inconsistency in the proof of the
DuBois--Reymond equations for the classical variational problem
with time delay recently obtained in \cite{FredericoTorres2012}:
the proof is correct if we suppose that
\begin{equation*}
\partial_4L[x,z]_\tau(t+\tau)\cdot \dot{x}(t)
+\partial_5L[x,z]_\tau(t+\tau)\cdot\ddot{x}(t)=0
\end{equation*}
along any extremal with time delay for all $t\in [a-\tau, b-\tau]$.
Such condition is trivially satisfied for the examples presented in
\cite{TorresTatianaGastao-2013,FredericoTorres2012}.
\end{remark}

Before presenting the extension of the famous Noether's first theorem
to variational problems of Herglotz type with time delay, we introduce
the definition of invariance and give two useful
necessary conditions for invariance.

\begin{definition}[Invariance with time delay]
\label{def invariance with time delay}
The one-parameter group of invertible $C^1$ transformations
\begin{equation}
\left\{
\begin{array}{c}
\bar{t}=\phi(t,x,\epsilon)=t+\sigma(t,x)\epsilon+o(\epsilon)\\
\bar{x}=\psi(t,x,\epsilon)=x+\xi(t,x)\epsilon+o(\epsilon)
\end{array} \label{global transf. delay}
\right.
\end{equation}
leave the functional $z$ defined by
\eqref{eq_Herg_delay}--\eqref{eq_initialvalue_Herg with delay}
invariant if
$$
\frac{d}{d\epsilon}\left[L\left(\bar{t},\bar{x}(\bar{t}),\frac{d\bar{x}}{d\bar{t}}(\bar{t}),
\bar{x}(\bar{t}-\tau),\frac{d\bar{x}}{d\bar{t}}(\bar{t}-\tau), \bar{z}(\bar{t})\right)
\cdot \frac{d\bar{t}}{dt}\right] \biggm\vert_{\epsilon=0}=0.
$$
\end{definition}

\begin{lemma}[Necessary condition for invariance with time delay I]
\label{invariance condt I}
If the functional $z$ defined by \eqref{eq_Herg_delay}--\eqref{eq_initialvalue_Herg with delay}
is invariant under the one-parameter group of transformations \eqref{global transf. delay},
then
$$
\frac{d\bar{z}}{d\epsilon}(t)\biggm\vert_{\epsilon=0}=0
$$
for each $t \in [a,b]$.
\end{lemma}

\begin{proof}
The proof is similar to the one of Lemma~\ref{lemma invariance}.
\end{proof}

The next result is a consequence of Lemma~\ref{invariance condt I}
and is useful in the proof of our Noether's first theorem
for variational problems of Herglotz with time delay.

\begin{lemma}[Necessary condition for invariance with time delay II]
\label{invariance condt II}
If the functional $z$ defined by \eqref{eq_Herg_delay}--\eqref{eq_initialvalue_Herg with delay}
is invariant under the one-parameter group of transformations \eqref{global transf. delay}, then
\begin{equation}
\label{invariance condt}
\begin{split}
\int_a^t\lambda(s)&\biggl[\partial_1L[x, z]_\tau(s)
\sigma(s)+\partial_2L[x, z]_\tau(s)\xi(s)
+\partial_3L[x, z]_\tau(s)(\dot{\xi}(s)
-\dot{x}(s)\dot{\sigma}(s)) \\
&+\partial_4L[x, z]_\tau(s)\xi(s-\tau)
+\partial_5 L[x, z]_\tau(s)\left(\dot{\xi}(s-\tau)
-\dot{x}(s-\tau)\dot{\sigma}(s-\tau)\right)\\
&+L[x, z]_\tau(s)\dot{\sigma}(s)\biggr]ds=0
\end{split}
\end{equation}
for each $t \in [a,b]$.
\end{lemma}

\begin{proof}
Since
$$
\frac{d\bar{z}}{dt}(t)
=L\left(\bar{t},\bar{x}(\bar{t}),\frac{d\bar{x}}{d\bar{t}}(\bar{t}),\bar{x}(\bar{t}
-\tau),\frac{d\bar{x}}{d\bar{t}}(\bar{t}-\tau), \bar{z}(\bar{t})\right)\cdot \frac{d\bar{t}}{dt}(t)
$$
and
$$
\frac{d\bar{t}}{dt}(t)\biggm\vert_{\epsilon=0}=1,
\quad \displaystyle\frac{d}{d\epsilon}\frac{d\bar{t}}{dt}(t)\biggm\vert_{\epsilon=0}
=\frac{d}{dt}\sigma(t,x),
$$
we get
\begin{equation*}
\frac{d}{d\epsilon}\left(\frac{d\bar{z}}{dt}\right)(t)\biggm\vert_{\epsilon=0}
=\frac{d L}{d\epsilon}\biggm\vert_{\epsilon=0}
\cdot\frac{d\bar{t}}{dt}(t)\biggm\vert_{\epsilon=0}
+L \cdot \frac{d}{d\epsilon}\frac{d\bar{t}}{dt}(t)\biggm\vert_{\epsilon=0}
=\frac{d L}{d\epsilon}\biggm\vert_{\epsilon=0}
+L \cdot \frac{d}{dt}\sigma(t,x).
\end{equation*}
Defining $h(t):=\displaystyle\frac{d\bar{z}}{d\epsilon}(t)\big\vert_{\epsilon=0}$,
\begin{multline}
\label{derivada}
\dot{h}(t)=
\displaystyle \partial_1L\frac{d\bar{t}}{d\epsilon}(t)\biggm\vert_{\epsilon=0}
+\partial_2L\frac{d\bar{x}}{d\epsilon}(t)\biggm\vert_{\epsilon=0}
+\partial_3L\frac{d}{d\epsilon}\frac{d\bar{x}}{d\bar{t}}(t)\biggm\vert_{\epsilon=0}
+\partial_4L\frac{d\bar{x}}{d\epsilon}(t-\tau)\biggm\vert_{\epsilon=0}\\
+ \displaystyle \partial_5L\frac{d}{d\epsilon}\frac{d\bar{x}}{d\bar{t}}(t-\tau)\biggm\vert_{\epsilon=0}
+\partial_6L\frac{d\bar{z}}{d\epsilon}(t)\biggm\vert_{\epsilon=0}+L\dot{\sigma}.
\end{multline}
Next we prove that
$$
\frac{d}{d\epsilon}\frac{d\bar{x}}{d\bar{t}}\biggm\vert_{\epsilon=0}
=\dot{\xi}-\dot{x}\dot{\sigma}.
$$
Because
$$
\frac{d\bar{x}}{dt}=\frac{d\bar{x}}{d\bar{t}}\cdot \frac{d\bar{t}}{dt}
=\frac{d\bar{x}}{d\bar{t}}\cdot\left(\frac{\partial\bar{t}}{\partial t}
+ \frac{\partial\bar{t}}{\partial x}\dot{x}\right),
$$
one has
\begin{equation}
\label{derivada 1}
\begin{split}
\frac{d}{d \epsilon}\frac{d\bar{x}}{dt}\biggm\vert_{\epsilon=0}
&= \frac{d}{d \epsilon}\left(\frac{d\bar{x}}{d\bar{t}}\cdot\left(
\frac{\partial\bar{t}}{\partial t}
+ \frac{\partial\bar{t}}{\partial x}\dot{x}\right)\right)\biggm\vert_{\epsilon=0}\\
&= \frac{d}{d \epsilon}\left(\frac{d\bar{x}}{d\bar{t}} \right)\biggm\vert_{\epsilon=0}
+\frac{d\bar{x}}{d\bar{t}}\biggm\vert_{\epsilon=0}
\cdot \frac{d}{d \epsilon}\left(\frac{\partial\bar{t}}{\partial t}
+ \frac{\partial\bar{t}}{\partial x}\dot{x} \right)\biggm\vert_{\epsilon=0}.
\end{split}
\end{equation}
On the other hand, since
\begin{equation*}
\frac{d}{d \epsilon}\frac{d\bar{x}}{dt}\biggm\vert_{\epsilon=0}
=\frac{d}{d \epsilon} \left(
\frac{\partial\bar{x}}{\partial t} + \frac{\partial\bar{x}}{\partial x}\dot{x}
\right)\biggm\vert_{\epsilon=0},
\end{equation*}
we get from equality \eqref{derivada 1} that
\begin{equation*}
\frac{\partial}{\partial t}\frac{d\bar{x}}{d\epsilon}\biggm\vert_{\epsilon=0}
+ \dot{x}\frac{\partial}{\partial x}\frac{d\bar{x}}{d\epsilon}\biggm\vert_{\epsilon=0}
= \frac{d}{d \epsilon}\frac{d\bar{x}}{d\bar{t}}\biggm\vert_{\epsilon=0}
+ \dot{x}\left(\frac{\partial\sigma}{\partial t}
+ \frac{\partial\sigma}{\partial x}\dot{x}\right)
\end{equation*}
and therefore
\begin{equation*}
\frac{\partial \xi}{\partial t} + \dot{x}
\frac{\partial \xi}{\partial x}
= \frac{d}{d \epsilon}\frac{d\bar{x}}{d\bar{t}}\biggm\vert_{\epsilon=0}
+ \dot{x}\dot{\sigma},
\end{equation*}
which is equilavent to
\begin{equation}
\label{derivada 4}
\frac{d}{d \epsilon}\frac{d\bar{x}}{d\bar{t}}\biggm\vert_{\epsilon=0}
= \dot{\xi}- \dot{x}\dot{\sigma}.
\end{equation}
Substituting  \eqref{derivada 4} into \eqref{derivada}, we get
\begin{multline*}
\dot{h}= \partial_1L\sigma+\partial_2L\xi+\partial_3L(\dot{\xi}
-\dot{x}\dot{\sigma})+\partial_4L\xi(t-\tau)\\
+\partial_5L(\dot{\xi}(t-\tau)-\dot{x}(t-\tau)\dot{\sigma}(t-\tau))
+\partial_6L h+L\dot{\sigma}.
\end{multline*}
Therefore, $h$ satisfies a first order differential equation whose solution is
\begin{multline*}
\lambda(t)h(t)-h(a)
= \int_a^t\lambda(s)\left[\partial_1L\sigma+\partial_2L\xi
+\partial_3L(\dot{\xi}-\dot{x}\dot{\sigma})+\partial_4L\xi(s-\tau)\right.\\
\left.+\partial_5L\left(\dot{\xi}(s-\tau)-\dot{x}(s-\tau)\dot{\sigma}(s-\tau)\right)
+L\dot{\sigma}\right]ds.
\end{multline*}
Finally, since functional $z$ defined by \eqref{eq_Herg_delay}--\eqref{eq_initialvalue_Herg with delay}
is invariant under the one-parameter group of transformations \eqref{global transf. delay},
$h\equiv 0$ by Lemma~\ref{invariance condt I} and we obtain \eqref{invariance condt}.
\end{proof}

The next result establishes an extension of the celebrated Noether first theorem
to variational problems of Herglotz type with time delay.

\begin{theorem}[Noether's first theorem for variational problems of Herglotz type with time delay]
\label{Thm Noether with delay}
If functional $z$ defined by \eqref{eq_Herg_delay}--\eqref{eq_initialvalue_Herg with delay}
is invariant under the one-parameter group of transformations \eqref{global transf. delay}, then
the quantities defined by
\begin{multline}\label{Quant 1}
\left[\lambda(t)\partial_3L[x,z]_\tau(t)+\lambda(t+\tau)\partial_5L[x,z]_\tau(t+\tau)\right]\xi(t)\\
+ \left[\lambda(t)L[x,z]_\tau(t)-\dot{x}(t)\left(\lambda(t)\partial_3L[x,z]_\tau(t)
+\lambda(t+\tau)\partial_5L[x,z]_\tau(t+\tau)\right)\right]\sigma(t)
\end{multline}
for $a\leq t\leq b-\tau$, and
\begin{equation}\label{Quant 2}
\lambda(t)\left[\partial_3L[x,z]_\tau(t)\xi(t) +(L[x,z]_\tau(t)
-\dot{x}(t)\partial_3L[x,z]_\tau(t))\sigma(t)\right]
\end{equation}
for $b-\tau \leq t \leq b$, are conserved along the generalized
extremals with time delay that satisfy
\begin{equation}
\label{extra-extra}
\partial_4L[x,z]_\tau(t+\tau)\cdot \dot{x}(t)
+\partial_5L[x,z]_\tau(t+\tau)\cdot\ddot{x}(t)=0
\end{equation}
for all $t\in [a-\tau, b-\tau]$, and
\begin{equation}
\label{Extra Hip Noeth}
\partial_4L[x,z]_\tau(t+\tau)\xi(t)+\partial_5L[x,z]_\tau(t+\tau)
\left(\dot{\xi}(t)-\dot{x}(t)\dot{\sigma}(t)\right)=0
\end{equation}
for all $t \in [a, b-\tau]$.
\end{theorem}

\begin{proof}
Suppose that the functional $z$ defined by \eqref{eq_Herg_delay}--\eqref{eq_initialvalue_Herg with delay}
is invariant  under the one-parameter group of transformations \eqref{global transf. delay} and that $x$
is a solution of the Euler--Lagrange equations \eqref{E-L delay 1}--\eqref{E-L delay 2}. From the necessary
condition for invariance with time delay II (Lemma \ref{invariance condt II}), we get that
\begin{multline*}
\int_a^t\lambda(s)\left[\partial_1L\sigma+\partial_2L\xi+
\partial_3L(\dot{\xi}-\dot{x}\dot{\sigma})+\partial_4L\xi(s-\tau)\right. \\
\left.+\partial_5L\left(\dot{\xi}(s-\tau)-\dot{x}(s-\tau)\dot{\sigma}(s-\tau)\right)+L\dot{\sigma}\right]ds=0
\end{multline*}
for each $t \in [a,b]$. Proceeding with a linear change of variable and noticing that we can assume
$\xi$ and $\sigma$ to be null outside $[a,b]$, the last equation is equivalent to
\begin{multline}
\label{eq inv}
\int_a^{t-\tau}\lambda(s)\left[\partial_1L\sigma+\partial_2L\xi
+\partial_3L(\dot{\xi}-\dot{x}\dot{\sigma})+L\dot{\sigma}\right]\\
+ \lambda(s+\tau)\left[\partial_4L(s+\tau)\xi+\partial_5L(s+\tau)
\left(\dot{\xi}(s)-\dot{x}(s)\dot{\sigma}(s)\right)\right]ds\\
+ \int_{t-\tau}^t\lambda(s)\left[\partial_1L\sigma+\partial_2L\xi
+\partial_3L(\dot{\xi}-\dot{x}\dot{\sigma})+L\dot{\sigma}\right]ds=0.
\end{multline}
Using hypothesis \eqref{Extra Hip Noeth}, equation \eqref{eq inv} implies that
\begin{align*}
\int_a^t \lambda(s)\left[\partial_1L\sigma+\partial_2L\xi
+\partial_3L(\dot{\xi}-\dot{x}\dot{\sigma})+L\dot{\sigma}\right]ds=0.
\end{align*}
From the arbitrariness of $t \in [a,b]$ we conclude that
\begin{align}
\label{eq null}
\partial_1L\sigma+\partial_2L\xi
+\partial_3L(\dot{\xi}-\dot{x}\dot{\sigma})+L\dot{\sigma}=0
\end{align}
for all $t \in [a,b]$. Then, equation \eqref{eq inv} becomes
\begin{align*}
\int_a^{t-\tau}&\Big(\lambda(s)\partial_1L\sigma+\left[\lambda(s)\partial_2L
+\lambda(s+\tau)\partial_4L(s+\tau)\right]\xi\\
& +\left[\lambda(s)\partial_3L+\lambda(s+\tau)\partial_5L(s+\tau)\right]\dot{\xi}\\
&+ \left[\lambda(s) L -\dot{x}\left(\lambda(s)\partial_3L
+\lambda(s+\tau)\partial_5L(s+\tau)\right)\right]\dot{\sigma}\Big)ds=0
\end{align*}
for $t \in[a+\tau,b]$. Using integration by parts, one has
\begin{align*}
&\int_a^{t-\tau}\Big(\lambda(s)\partial_1L\sigma+\left[\lambda(s)\partial_2L+\lambda(s+\tau)
\partial_4L(s+\tau)\right]\xi\\
& \qquad - \frac{d}{ds}\left[\lambda(s)\partial_3L+\lambda(s+\tau)\partial_5L(s+\tau)\right]\xi\\
& \qquad -\frac{d}{ds} \left[\lambda(s) L -\dot{x}\left(\lambda(s)\partial_3L
+\lambda(s+\tau)\partial_5L(s+\tau)\right)\right]\sigma \Big)ds\\
& + \Big[\left(\lambda(s)\partial_3L+\lambda(s+\tau)\partial_5L(s+\tau)\right)\xi\\
& \qquad +\left(\lambda(s) L -\dot{x}\left(\lambda(s)\partial_3L+\lambda(s+\tau)
\partial_5L(s+\tau)\right)\right)\sigma\Big]_a^{t-\tau} = 0.
\end{align*}
Observe that the terms in $\xi$ inside the integral are null
because $x$ satisfies the Euler--Lagrange equation
on $[a, b-\tau]$ and that, from the DuBois--Reymond equation \eqref{DuBois-R (1)},
the sum of the remaining terms of the integral is zero. This leads to
\begin{multline*}
\Big[\left(\lambda(s)\partial_3L+\lambda(s+\tau)\partial_5L(s+\tau)\right)\xi\\
+\left(\lambda(s) L -\dot{x}\left(\lambda(s)\partial_3L+\lambda(s+\tau)
\partial_5L(s+\tau)\right)\right)\sigma\Big]_a^{t-\tau} = 0
\end{multline*}
for every $t \in [a+\tau,b]$, which means that
\begin{equation*}
\Big(\lambda(s)\partial_3L+\lambda(t+\tau)\partial_5L(t+\tau)\Big)\xi
+\Big(\lambda(s) L -\dot{x}\left(\lambda(s)\partial_3L
+\lambda(t+\tau)\partial_5L(t+\tau)\right)\Big)\sigma
\end{equation*}
is constant for $t \in [a,b-\tau]$. Consider $[t_1,t_2] \subseteq [b-\tau,b]$.
From equation \eqref{eq null} one has
\begin{align*}
& \int_{t_1}^{t_2}(\lambda(s)\partial_1L\sigma
+\lambda(s)\partial_2L\xi+\lambda(s)\partial_3L\dot{\xi}
+\lambda(s)(L-\dot{x}\partial_3L)\dot{\sigma})ds=0.
\end{align*}
Using integration by parts, we get
\begin{multline*}
\int_{t_1}^{t_2}(\lambda(s)\partial_1L\sigma
+\lambda(s)\partial_2L\xi-\frac{d}{ds}\left(\lambda(s)\partial_3L\right)\xi
-\frac{d}{ds}\left[\lambda(s)(L-\dot{x}\partial_3L)\right]\sigma) ds\\
+\left[\lambda(s)\partial_3L\xi+\lambda(s)(L-\dot{x}\partial_3L)
\sigma\right]_{t_1}^{t_2}=0.
\end{multline*}
Observe that the terms in $\xi$ inside the integral are null because $x$
satisfies the Euler--Lagrange equation \eqref{E-L delay 2}
and that, from the DuBois--Reymond equation \eqref{DuBois-R (2)},
the sum of the remaining terms of the integral is zero. This leads to
\begin{equation*}
\left[\lambda(s)\partial_3L\xi+\lambda(s)(L-\dot{x}\partial_3L)\sigma\right]_{t_1}^{t_2}=0.
\end{equation*}
From the arbitrariness of $t_1, t_2 \in [b-\tau,b]$, we conclude that
\begin{equation*}
\lambda(s)\partial_3L\xi+\lambda(s)(L-\dot{x}\partial_3L)\sigma
\end{equation*}
is constant in $[b-\tau,b]$. This ends the proof of our main result.
\end{proof}

\begin{remark}
In the classical variational problem and in the variational problem of Herglotz,
the hypotheses \eqref{extra-extra} and \eqref{Extra Hip Noeth} are trivially satisfied.
\end{remark}

\begin{remark}
Our first Noether-type theorem is a generalization of Noether's first theorem
for the classical variational problem of Herglotz type presented in \cite{Georgieva2002},
that is, Theorem~\ref{1st Noether THM} is a corollary of Theorem~\ref{Thm Noether with delay}.
\end{remark}

Our results provide generalizations of the variational results with
time delay presented in \cite{FredericoTorres2012}. If the Lagrangian $L$
in the definition of $z$, \eqref{eq_Herg_delay}, does not depend on $z$,
then $\partial_6L \equiv 0$ and $\lambda(t)\equiv 1$. In that case, problem
$(P)$ reduces to the classical variational problem with time delay.
The Euler--Lagrange equations, the DuBois--Reymond conditions and Noether's first
theorem with time delay obtained by Frederico and Torres in
\cite{FredericoTorres2012} are particular cases of
Theorem~\ref{Thm E-L eq time delay}, Theorem~\ref{Th. DuBois-R time delay}
and Theorem \ref{Thm Noether with delay}, respectively.
In what follows we use the notation
$$
[x]_\tau:=\left(t,x(t),\dot{x}(t),x(t-\tau), \dot{x}(t-\tau)\right).
$$

\begin{corollary}[See \cite{FredericoTorres2012}]
\label{corol E-L delay}
If $x$ is an extremizer to the functional
\begin{equation}
\label{eq:prb:ft12}
\int_a^bL(t,x(t),\dot{x}(t),x(t-\tau), \dot{x}(t-\tau))dt,
\end{equation}
then $x$ satisfies the Euler--Lagrange equations
\begin{equation}
\label{corol E-L delay 1}
\partial_4L[x]_\tau(t+\tau)-\frac{d}{dt}\partial_5L[x]_\tau(t+\tau)
+\partial_2L[x]_\tau(t)-\frac{d}{dt}\partial_3L[x]_\tau(t)=0,
\end{equation}
$a\leq t \leq b-\tau$, and
\begin{equation}
\label{corol E-L delay 2}
\partial_2L[x]_\tau(t)-\frac{d}{dt}\partial_3L[x]_\tau(t) =0,
\end{equation}
$b-\tau \leq t \leq b$.
\end{corollary}

\begin{corollary}[Cf. \cite{FredericoTorres2012}]
\label{corol DuBois}
If $x$ is an extremizer to the functional \eqref{eq:prb:ft12} and
$$
\partial_4L[x]_\tau(t+\tau)\cdot \dot{x}(t) +\partial_5L[x]_\tau(t+\tau)
\cdot\ddot{x}(t)=0,
$$
$t\in[a-\tau,b-\tau]$, then $x$ satisfies the DuBois--Reymond equations
\begin{equation*}
\frac{d}{dt}\left\{L[x]_\tau(t) - \dot{x}(t)\left[\partial_3
L[x]_\tau(t) + \partial_5L[x]_\tau(t+\tau)\right]\right\}
=\partial_1L[x]_\tau(t),
\end{equation*}
$a \leq t\leq b-\tau$, and
\begin{equation*}
\frac{d}{dt}\left\{L[x]_\tau(t)-\dot{x}(t)\partial_3L[x]_\tau(t)\right\}
=\partial_1L[x]_\tau(t),
\end{equation*}
$b-\tau \leq t \leq b$.
\end{corollary}

\begin{corollary}[Cf. \cite{FredericoTorres2012}]
\label{corol Noether delay}
If functional \eqref{eq:prb:ft12} is invariant in the sense
of Definition~\ref{def invariance}, then the quantities
\begin{multline*}
\left[\partial_3L[x]_\tau(t)+\partial_5L[x]_\tau(t+\tau)\right]\xi(t)\\
+ \left[L[x]_\tau(t)-\dot{x}(t)\left(
\partial_3L[x]_\tau(t)+\partial_5L[x]_\tau(t+\tau)\right)\right]\sigma(t),
\end{multline*}
$a\leq t\leq b-\tau$, and
\begin{equation*}
\partial_3L[x]_\tau(t)\xi(t)+[L[x]_\tau(t)-\dot{x}(t)\partial_3L[x]_\tau(t)]\sigma(t),
\end{equation*}
$b-\tau \leq t \leq b$, are conserved along the solutions of the Euler--Lagrange
equations \eqref{corol E-L delay 1}--\eqref{corol E-L delay 2} that satisfy
$$
\partial_4L[x]_\tau(t+\tau)\cdot \dot{x}(t)
+\partial_5L[x]_\tau(t+\tau)\cdot\ddot{x}(t)=0,
$$
$t\in[a-\tau,b-\tau]$, and
$$
\partial_4L[x]_\tau(t+\tau)\xi(t)+\partial_5L[x]_\tau(t+\tau)
\left(\dot{\xi}(t)-\dot{x}(t)\dot{\sigma}(t)\right)=0,
$$
$t \in [a, b-\tau]$.
\end{corollary}


\section{Illustrative example}
\label{sec:ex}

We present an example that shows the usefulness of our results.
Consider the following Herglotz's variational problem with time delay $\tau=1$:
\begin{equation}
\label{example}
\begin{gathered}
z(2)\longrightarrow \textrm{extr}\\
\dot{z}(t)=L[x,z]_1(t):=\left(\dot{x}(t-1)\right)^2 + z(t), \quad t\in[0,2],\\
x(t)=-t, \quad t \in [-1,0],\\
x(2)=1, \quad z(0)=0.
\end{gathered}
\end{equation}
For this problem, Euler--Lagrange optimality conditions
\eqref{E-L delay 1}--\eqref{E-L delay 2} given by
Theorem~\ref{Thm E-L eq time delay} assert that
\begin{equation*}
\begin{cases}
\dot{x}(t)-\ddot{x}(t)=0, \quad t \in [0,1],\\
0=0, \quad t \in [1,2].
\end{cases}
\end{equation*}
Solving the equation of previous system
with the initial condition $x(0)=0$, we obtain
\begin{equation*}
x(t)=-k+ke^t, \quad t \in [0,1],
\end{equation*}
for some constant $k \in \mathbb{R}$. Since in $[0,1]$ $z$ is
defined by $\dot{z}(t)=1+z(t)$, with $z(0)=0$, we obtain
\begin{equation*}
z(t)=e^t-1, \quad t \in [0,1].
\end{equation*}
In order to illustrate our remaining results
(Theorems~\ref{Th. DuBois-R time delay} and \ref{Thm Noether with delay}),
we look for trajectories $x$ that satisfy hypothesis \eqref{Extra Hip}:
$2\dot{x}(t)\cdot \ddot{x}(t)=0$, $t \in [-1,1]$. This condition is trivially
satisfied in the interval $[-1,0]$, but leads to $k=0$ and, consequently,
$x(t)=0$ in $[0,1]$. Hence, we get a family $x_\phi$ of generalized extremals
with time delay given by
\begin{equation}
\label{example_x}
x_\phi(t)=
\begin{cases}
-t, \quad t \in[-1,0],\\
0, \quad t \in[0,1],\\
\phi(t), \quad t \in[1,2],\\
1, \quad t=2,\\
\end{cases}
\end{equation}
where the continuous function $\phi$ is chosen to guarantee that $x_\phi$ is a $C^2$ function.
With $x$ defined by \eqref{example_x} for some $\phi$, and $z$ defined in $[1,2]$
as $\dot{z}(t)=z(t)$ with $z(1)=e-1$, it follows that $z(t)=e^{t-1}(e-1)$
for $t \in [1,2]$ and, consequently,
\begin{equation}
\label{example_z}
z(t)=
\begin{cases}
e^t-1, \quad t \in [0,1],\\
e^{t-1}(e-1), \quad t \in [1,2],
\end{cases}
\end{equation}
for which $z(2)=e^2-e$. Next we show that DuBois--Reymond conditions
\eqref{DuBois-R (1)}--\eqref{DuBois-R (2)} given by
Theorem~\ref{Th. DuBois-R time delay} are valid for
$x$ and $z$  given by \eqref{example_x}--\eqref{example_z}.
In this case, \eqref{DuBois-R (1)} reduces to
\begin{equation*}
\frac{d}{dt}\left[\lambda(t)\left(\dot{x}^2(t-1)+z(t)\right)
-\dot{x}(t)\left(2\lambda(t+1)\dot{x}(t)\right)\right]=0,
\quad t \in [0,1],
\end{equation*}
which is equivalent to
\begin{equation*}
\frac{d}{dt}\left[\lambda(t)e^t\right]=0, \quad t \in [0,1].
\end{equation*}
Since $\lambda(t)=e^{-t}$, condition \eqref{DuBois-R (1)} holds for
$t \in [0,1]$. Similarly, it can be proved that condition
\eqref{DuBois-R (2)} holds for $t \in [1,2]$.
Finally, we show the relevance of our main result (Theorem~\ref{Thm Noether with delay}).
First we define a one-parameter group of transformations on $t$ and $x$ with generators
$\sigma(t) \equiv 1$ and $\xi(t) \equiv 0$, respectively. Since the Lagrangian defined in \eqref{example}
is autonomous, i.e., does not depend explicitly on $t$, then it is invariant in the sense
of Definition~\ref{def invariance with time delay}.
Observe that in this case hypothesis \eqref{Extra Hip Noeth} is trivially satisfied.
Theorem~\ref{Thm Noether with delay} asserts that \eqref{Quant 1} and \eqref{Quant 2}
are constant in $t$, in intervals $[0,1]$ and $[1,2]$, respectively,
along generalized extremals with time delay that satisfy hypotheses
\eqref{extra-extra} and \eqref{Extra Hip Noeth}.
Observe that \eqref{Quant 1} is equal to
\begin{equation*}
\begin{split}
\left[\lambda(t)L[x,z]_1(t)-2 \left(\dot{x}(t)\right)^2\lambda(t+1)\right]\sigma(t)
&=e^{-t}\left[\dot{x}^2(t-1)+z(t)\right]\\
&=e^{-t}\left[1+e^{t}-1\right], \quad t \in [0,1],
\end{split}
\end{equation*}
which is equal to one. Similarly, it can be easily proved that quantity
\eqref{Quant 2} is also constant in $t$ and  equal to $1-e^{-1}$.


\section*{Acknowledgements}

This work was supported by Portuguese funds through the
\emph{Center for Research and Development in Mathematics and Applications} (CIDMA)
and the \emph{Portuguese Foundation for Science and Technology}
(``FCT---Funda\c{c}\~ao para a Ci\^encia e a Tecnologia''),
within project UID/MAT/04106/2013.
Torres was also supported by EU funding under
the 7th Framework Programme FP7-PEOPLE-2010-ITN, grant agreement 264735-SADCO;
and by the FCT project OCHERA, PTDC/EEI-AUT/1450/2012, co-financed by
FEDER under POFC-QREN with COMPETE reference FCOMP-01-0124-FEDER-028894.
The authors are very grateful to two anonymous referees,
for several constructive remarks and questions.



\medskip

Received May 2014; revised October 2014; accepted January 2015.

\medskip


\end{document}